\documentclass[11pt]{amsart}
\usepackage{graphicx}
\usepackage{amscd}
\usepackage{amsmath}
\usepackage{amsxtra}
\usepackage{amsfonts}
\usepackage{amssymb}
\usepackage{xcolor}
\usepackage{mathrsfs}

\newtheorem{theorem}{Theorem}[section]
\newtheorem{corollary}[theorem]{Corollary}
\newtheorem{lemma}[theorem]{Lemma}
\newtheorem{proposition}[theorem]{Proposition}

\newtheorem{conjecture}[theorem]{Conjecture}
\theoremstyle{definition}
\newtheorem{definition}[theorem]{Definition}

\newtheorem{remark}[theorem]{Remark}

\theoremstyle{remark}

\renewcommand{\theclaim}{\textup{\theclaim}}

\numberwithin{equation}{section}

\def\openone

{\mathchoice

{\hbox{\upshape \small1\kern-3.3pt\normalsize1}}

{\hbox{\upshape \small1\kern-3.3pt\normalsize1}}

{\hbox{\upshape \tiny1\kern-2.3pt\SMALL1}}

{\hbox{\upshape \Tiny1\kern-2pt\tiny1}}}

\makeatletter

\newbox\ipbox

\newcommand{\ip}[2]{\left\langle #1\, , \,#2\right\rangle}
\newcommand{\diracb}[1]{\left\langle #1\mathrel{\mathchoice

{\setbox\ipbox=\hbox{$\displaystyle \left\langle\mathstrut
#1\right.$}

\vrule height\ht\ipbox width0.25pt depth\dp\ipbox}

{\setbox\ipbox=\hbox{$\textstyle \left\langle\mathstrut
#1\right.$}

\vrule height\ht\ipbox width0.25pt depth\dp\ipbox}

{\setbox\ipbox=\hbox{$\scriptstyle \left\langle\mathstrut
#1\right.$}

\vrule height\ht\ipbox width0.25pt depth\dp\ipbox}

{\setbox\ipbox=\hbox{$\scriptscriptstyle \left\langle\mathstrut
#1\right.$}

\vrule height\ht\ipbox width0.25pt depth\dp\ipbox}

}\right. }

\newcommand{\dirack}[1]{\left. \mathrel{\mathchoice

{\setbox\ipbox=\hbox{$\displaystyle \left.\mathstrut
#1\right\rangle$}

\vrule height\ht\ipbox width0.25pt depth\dp\ipbox}

{\setbox\ipbox=\hbox{$\textstyle \left.\mathstrut
#1\right\rangle$}

\vrule height\ht\ipbox width0.25pt depth\dp\ipbox}

{\setbox\ipbox=\hbox{$\scriptstyle \left.\mathstrut
#1\right\rangle$}

\vrule height\ht\ipbox width0.25pt depth\dp\ipbox}

{\setbox\ipbox=\hbox{$\scriptscriptstyle \left.\mathstrut
#1\right\rangle$}

\vrule height\ht\ipbox width0.25pt depth\dp\ipbox}

} #1\right\rangle}

\newcommand{\beq}{\begin{equation}}

\newcommand{\eeq}{\end{equation}}

\newcommand{\cj}[1]{\overline{#1}}

\newcommand{\bz}{\mathbb{Z}}

\newcommand{\br}{\mathbb{R}}
\newcommand{\bc}{\mathbb{C}}

\newcommand{\bn}{\mathbb{N}}

\def\blfootnote{\xdef\@thefnmark{}\@footnotetext}

\input xy
\xyoption{all}
\usepackage{amssymb}

\def\-{^{-1}}

\def\D{\mathscr{D}}

\def\ty{\emptyset}

\begin{document}

\title[The momentum operator on a union of intervals]{The momentum operator on a union of intervals and the Fuglede conjecture}

\author{Dorin Ervin Dutkay$^*$}
\address{[Dorin Ervin Dutkay] University of Central Florida\\
	Department of Mathematics\\
	4000 Central Florida Blvd.\\
	P.O. Box 161364\\
	Orlando, FL 32816-1364\\
U.S.A.\\} \email{Dorin.Dutkay@ucf.edu}
\thanks{$^*$Corresponding author.}

\author{Palle E.T. Jorgensen}
\address{[Palle E.T. Jorgensen]University of Iowa\\
Department of Mathematics\\
14 MacLean Hall\\
Iowa City, IA 52242-1419\\}\email{palle-jorgensen@uiowa.edu}

\subjclass[2010]{47E05,42A16}
\keywords{momentum operator, self-adjoint extension, Fourier bases, Fuglede conjecture }

\begin{abstract}
The purpose of the present paper is to place a number of geometric (and hands-on) configurations relating to spectrum and geometry inside a general framework for the {\it Fuglede conjecture}. Note that in its general form, the Fuglede conjecture concerns general Borel sets $\Omega$ in a fixed number of dimensions $d$ such that  $\Omega$ has finite positive Lebesgue measure. The conjecture proposes a correspondence between two properties for $\Omega$, one takes the form of spectrum, while the other refers to a translation-tiling property. We focus here on the case of dimension one, and the connections between the Fuglede conjecture and properties of the self-adjoint extensions of the momentum operator $\frac{1}{2\pi i}\frac{d}{dx}$, realized in $L^2$ of  a union of intervals.

\end{abstract}
\maketitle
\tableofcontents
\newcommand{\Ds}{\mathsf{D}}
\newcommand{\Dmax}{\mathscr D_{\operatorname*{max}}}
\newcommand{\Dmin}{\Ds_{\operatorname*{min}}}

\section{Introduction}

 For bounded open domains $\Omega$ in $\br^d$, the Fuglede problem deals with two properties that $\Omega$ may or may not have, one (called spectral) is relative to the Hilbert space $L^2(\Omega)$, the question of whether $L^2(\Omega)$  has an orthogonal $d$-variable Fourier basis, and the other is geometric (tiling), whether $\Omega$ tiles $\br^d$ by some set of translation vectors. The original problem asked whether the two properties are equivalent.

In this paper we show how tools from operator theory (especially choices of spectral representations for unbounded operators), serve to link the two sides of the problem, spectrum vs geometry.

    Since the inception, stated this way, the Fuglede conjecture is now known to be negative, more precisely that the two properties are not equivalent in dimension 3 and higher \cite{Tao04,MR2159781,MR2237932,MR2264214,MR2267631}. Nonetheless, the Fuglede problem is even open for $d=1$.

Parallel to this we note that there are many closely related new research directions, including analysis on fractals, which deal with various notions of interplay between spectral theoretic properties on the one hand, and geometry on the other, e.g., direct problems and inverse problems. We further stress that the original formulation was stated in terms of properties for the set of $d$ partial derivative operators for the coordinate directions in $\Omega$, specifically the possible extensions of partial derivative operators in the form of commuting generators for unitary one-parameter groups acting in $L^2(\Omega)$. Such extensions are known to necessarily be local translation generators. Moreover, following quantum theory, such generators may be viewed as momentum operators, a viewpoint motivated by the canonical duality from quantum mechanics for momentum and position observables. This formulation in turn makes a direct connection to scattering theoretic properties, again related to the Hilbert space $L^2(\Omega)$. And in this form, the problem is of interest even for $d=1$; and so the case when $\Omega$ is a union of intervals.

   Continuing earlier work (e.g., \cite{DuJo15}) we aim here at presenting new results for the $d=1$ Fuglede problem, and making the presentation as self-contained as possible, for the readers who might not be experts in this field. 
	
	 Starting with its classical roots, the Fuglede problem addresses two related properties for bounded domains $\Omega$ in $\br^d$ . More precisely, Fuglede's question asks for a specific linking between multivariable spectra on one side, and geometry of $\Omega$  on the other (spectral vs tiling). But by now, the Fuglede problem/conjecture has become distinctly interdisciplinary. It has come to encompass a diverse variety of neighboring fields of mathematics each of which in turn lies at the cross-roads of at least the following six separate disciplines: (i) harmonic analysis, (ii) spectral and scattering theory for operators in Hilbert space, (iii) metric/convex geometry, (iv) fractals, (v) operator algebras, and (vi) representation theory. For the benefit of readers, we include citations to the following list of papers, each dealing with one or the other of the above six areas, \cite{MR3838440,MR4458124,MR4448683,MR4446811,MR4305973,MR4177387,MR4093913,MR4039387,MR3678495,MR3649367,MR2576285}.

	The paper is organized as follows: in Section 1, we introduce some definitions related to unbounded symmetric operators and their extensions, the associated one-parameter unitary group, and we recall Fuglede's result and conjecture, which serve as the main motivation for our paper. In the next sections we focus on the case when the set $\Omega$ is a finite union of intervals in dimension $d=1$. In Section 2, we study the symmetric and the self-adjoint extensions $A$ of the momentum operator $\Ds=\frac{1}{2\pi i}\frac{d}{dx}$ on the space $C_c^\infty(\Omega)$ of infinitely differentiable functions with compact support in $\Omega$. In Section 3, we describe the spectral decomposition of such self-adjoint extensions. Section 4 is devoted to the one-parameter unitary group $U(t)=\exp(2\pi itA)$, $t\in\br$, which acts as translations inside the intervals of $\Omega$ and has a different behavior at the end-points. In Section 5, we make the connections between the existence of orthogonal Fourier bases on $\Omega$ and the properties of the self-adjoint extensions $A$ or of the unitary group $U(t)$.

\section{Notations and preliminaries}

When $d = 1$, with a choice of an open set $\Omega$,  the corresponding connected components will then be intervals, and so $\Omega$ takes the form of a finite union of intervals as follows.

\begin{definition}\label{defi2.1}

Let 
$$\Omega=\bigcup_{i=1}^n(\alpha_i,\beta_i),\mbox{ where }-\infty<\alpha_1<\beta_1<\alpha_2<\beta_2<\dots<\alpha_n<\beta_n<\infty.$$
So 
$$\Omega=\bigcup_{i=1}^nJ_i,\mbox{ where }J_i=(\alpha_i,\beta_i)\mbox{ for all }i\in\{1,\dots,n\}.$$
On $\Omega$ we consider the Lebesgue measure $dx$. We denote by $\partial\Omega$ the boundary of $\Omega$,
$$\partial\Omega=\{\alpha_i,\beta_i : i\in\{1,\dots,n\}\}.$$
For a function $f$ on $\partial\Omega$ we use the notation 
$$\int_{\partial\Omega}f=\sum_{i=1}^n(f(\beta_i)-f(\alpha_i)).$$
Consider the subspace of infinitely differentiable compactly supported functions on $\Omega$, $C_c^\infty(\Omega)$. We define the {\it differential/momentum operator} $\Ds$ on $C_c^\infty(\Omega)$:

$$\Ds f=\frac{1}{2\pi i}f',\quad (f\in C_c^\infty(\Omega)).$$

Define also the subspaces

\begin{multline}\label{eqd_0}
\mathscr D_0(\Omega)=\left\{f:\Omega\rightarrow\bc : f\mbox{ is absolutely continuous on each }J_i,\right.\\
\left. f(\alpha_i+)=f(\beta_i-)=0\mbox{ for all }i\mbox{ and }f'\in L^2(\Omega)\right\},
\end{multline}

\begin{equation}
\Dmax=\left\{f:\Omega\rightarrow\bc : f\mbox{ is absolutely continuous on each interval $J_i$ and $f'\in L^2(\Omega)$}\right\},
\label{eq1.1.1}
\end{equation}

\end{definition}

\begin{remark}\label{rem1.1}
If a function $f$ is absolutely continuous on each interval $J_i$ and $f'\in L^2(\Omega)$, then the values of the function $f$ at the endpoints $\alpha_i$ and $\beta_i$ are well defined. Indeed, fix $i\in\{1,\dots, n\}$ and a point $x_0\in (\alpha_i,\beta_i)$. Then, since $f$ is absolutely continuous and with $f'\in L^2(\Omega)\subset L^1(\Omega)$, one has 
$$f(x)=f(x_0)+\int_{x_0}^x f'(t)\,dt,\mbox{ for all } x\in (\alpha_i,\beta_i),$$
and therefore 
$$f(\alpha_i+)=f(x_0)-\int_{\alpha_i}^{x_0} f'(t)\,dt\mbox{ and, }f(\beta_i-)=f(x_0)+\int_{x_0}^{\beta_i} f'(t)\,dt.$$
This means that we can define $f(\alpha_i):=f(\alpha_i+)$, and similarly $f(\beta_i)$ by continuity.

For $f\in\Dmax$, we denote by $f(\vec\alpha)=\left(f(\alpha_1),\dots, f(\alpha_n)\right)$ and similarly for $f(\vec\beta)$.

\end{remark}

We are looking for closed symmetric and for self-adjoint extensions of the operator $\Ds$ on $C_c^\infty(\Omega)$.

\begin{definition}\label{def2.1}

We recall some notions about unbounded linear operators, see for example \cite[Chapter X]{Con90}. Let $H$ be a Hilbert space.

\begin{itemize}
	\item[$\bullet$] We denote by $\mathscr B(H)$ the set of bounded linear operators on $H$. 
	
	\item We denote the domain of an unbounded operator $T$ by $\mathscr D(T)$. 
	
	\item[$\bullet$] An operator $T_2$ is an {\it extension} of $T_1$ if $\D(T_2)$ contains $\D(T_1)$ and $T_2f=T_1 f$, for all $f\in \D(T_1)$. We write $T_1\subseteq T_2$.  
	
	\item[$\bullet$] An operator $T$ is {\it closed} if its graph is closed, i.e., if $\{f_n\}$ in $\D(T)$ converges to $f$ and $\{Tf_n\}$ converges to $g$ then $f\in\D(T)$ and $Tf=g$.

	\item[$\bullet$] For a densely defined unbounded operator $T:H_1\rightarrow H_2$, the {\it adjoint} operator $T^*:H_2\rightarrow H_1$ is defined on the set of vectors $g\in H_2$ with the property that the linear functional $\mathscr D(T)\ni f\mapsto \ip{Tf}{g}$ is bounded. In this case, by Riesz's lemma, there exists a unique element $T^*g\in H_1$ such that $\ip{Tf}{g}=\ip{f}{T^*g}$ for all $f\in\mathscr D(T)$.
	
	\item[$\bullet$]  An densely defined operator is called {\it symmetric} if 
$$\ip{Tf}{g}=\ip{f}{Tg},\mbox{ for all }f,g\in \D(T).$$

	\item[$\bullet$] An operator is called {\it self-adjoint} if $\D(T^*)=\D(T)$ and $T^*f=Tf$ for all $f\in \D(T)$.
	
\item An operator $N$ is called {\it normal} if it is closed, densely defined and $N^*N=NN^*$.
	
\end{itemize}

\end{definition}

\begin{definition}\label{defsp0}
Recall that a (possibly unbounded) linear operator $T:\mathcal H\rightarrow\mathcal K$ is called {\it boundedly invertible} if there is a bounded operator $S:\mathcal K\rightarrow\mathcal H$ such that $TS=I$ and $ST\subseteq I$. The {\it resolvent set} $\rho(T)$ for the operator $T$ is defined by 
$$\rho(T)=\{\lambda\in\bc : \lambda I-T \mbox{ is boundedly invertible}\}.$$
The {\it spectrum} of $T$ is defined as $\sigma(T)=\bc\setminus\rho(T)$. 

\end{definition}

\begin{definition}\label{defspectral}
If $X$ is a set, $\mathcal B$ is a $\sigma$-algebra of subsets of $X$, and $H$ is a Hilbert space, a {\it spectral measure/spectral resolution} for $(X,\mathcal B,H)$ is a function $E:\mathcal B\rightarrow\mathscr B(H)$ such that:
\begin{itemize}
	\item[(a)] for each $\Delta$ in $\mathcal B$, $E(\Delta)$ is a projection;
	\item[(b)] $E(\ty)=0$ and $E(X)=I_H$;
	\item[(c)] $E(\Delta_1\cap\Delta_2)=E(\Delta_1)E(\Delta_2)$ for $\Delta_1$ and $\Delta_2$ in $\mathcal B$;
	\item[(d)] if $\{\Delta_n\}_{n=1}^\infty$ are pairwise disjoint sets from $\mathcal B$, then 
	$$E\left(\bigcup_{n=1}^\infty\Delta_n\right)=\sum_{n=1}^\infty E(\Delta_n).$$
\end{itemize}

\end{definition}

\begin{theorem}\label{thspectral}{\bf Spectral theorem for unbounded self-adjoint operators.} If $A$ is a self-adjoint operator on $H$, then there exists a unique spectral measure $E$ defined on the Borel subsets of $\br$ such that 
\begin{itemize}
	\item[(a)] $A=\int x\,dE(x)$;
	\item[(b)] $E(\Delta)=0$ if $\Delta\cap\sigma(A)=\ty$;
	\item[(c)] if $U$ is an open subset of $\br$ and $U\cap\sigma(A)\neq\ty$, then $E(U)\neq 0$;
	\item[(d)] if $B\in\mathscr B(H)$ such that $BA\subseteq AB$, then $B(\int\phi\,dE)\subseteq (\int \phi\,dE)B$ for every Borel function $\phi$ on $\br$. 
\end{itemize}

\end{theorem}

\begin{remark}\label{remspec}

 Note that, starting with a fixed selfadjoint operator $A$ in a Hilbert space $H$, the corresponding projection valued measure $E$ then induces a spectral representation for $A$  via a system of scalar measures indexed by the vectors in $H$. These are defined for $h\in H$ by 
\begin{equation}
E_{h,h}(\Delta)=\ip{E(\Delta)h}{h},\mbox{ for all Borel sets }\Delta.
\label{eqremspec0}
\end{equation}

Let $A$ be a self-adjoint operator in a Hilbert space $H$ with spectral resolution $E$; see Definition \ref{defspectral}. Then the following three conclusions follow immediately :
\begin{enumerate}
	\item If $\varphi:\br\rightarrow \bc$ is measurable, then the operator 
	\begin{equation}
	\varphi(A):=\int_{\br}\varphi(x)\,dE(x)
	\label{eqremspec1}
	\end{equation}	
	is well defined in $H$ and normal. 
	\item The dense domain of $\varphi(A)$, denoted $\D(\varphi(A))$, consists of all vectors $h\in H$ such that 
	\begin{equation}
	\int_\br |\varphi(x)|^2\,dE_{h,h}<\infty,
	\label{eqremspec2}
	\end{equation}
	\item For all $h\in \D(\varphi(A))$, 
	\begin{equation}
	\|\varphi(A)h\|^2=\int_{\br}|\varphi(x)|^2\,dE_{h,h}.
	\label{eqremspec3}
	\end{equation}
\end{enumerate}

\end{remark}

\begin{definition}\label{defstone}
A {\it strongly continuous one parameter unitary group} is a function $U:\br\rightarrow\mathscr B(H)$ such that for all $s$ and $t$ in $\br$:
\begin{itemize}
	\item[(a)] $U(t)$ is a unitary operator;
	\item[(b)] $U(s+t)=U(s)U(t)$;
	\item[(c)] if $h\in H$ and $t_0\in\br$ then $U(t)h\rightarrow U(t_0)h$ as $t\rightarrow t_0$.
\end{itemize}

\end{definition}

\begin{theorem}\label{thstone0}
Let $A$ be a self-adjoint operator on $H$ and let $E$ on $(X,\mathcal B,H)$ be its spectral measure. Define 
$$U(t)=\exp(2\pi it A)=\int e^{2\pi itx}\,d E(x),\quad (t\in\br).$$
Then 
\begin{itemize}
	\item[(a)] $(U(t))_{t\in\br}$ is a strongly continuous one parameter group;
	\item[(b)] if $h\in\mathscr D(A)$, then 
	$$\lim_{t\rightarrow 0}\frac{1}{t}(U(t)h-h)=2\pi i Ah;$$
	\item[(c)] if $h\in H$ and $\lim_{t\rightarrow 0}\frac1t(U(t)h-h)$ exists, then $h\in\mathscr D(A)$. Consequently, $\D(A)$ is invariant under $U(t)$.
\end{itemize}

\end{theorem}

\begin{theorem}\label{thstone} {\bf Stone's Theorem.} 
If $U$ is a strongly continuous one parameter unitary group, then there exists a self-adjoint operator $A$ such that $U(t)=\exp(2\pi itA)$, $t\in \br$, and conversely. The self-adjoint operator $A$ is called the infinitesimal generator for $U$. 

\end{theorem}

We will also need the following well-known lemma about multiplication operators.

\begin{lemma}\label{lemfu3}
Let $(X,\mu)$ be a $\sigma$-finite measure space and let $\phi:X\rightarrow\bc$ be a measurable function. Let $\D=\{f\in L^2(\mu) : \phi f\in L^2(\mu)\}$ and define $Af=\phi f$ for all $f\in\D$. Then $A$ is a closed operator, $\D(A^*)=\D$, and $A^*f=\cj \phi f$ for $f\in \D$. In particular, if $f$ is real-valued, then $A$ is self-adjoint. 
\end{lemma}

\begin{proof}
First, the domain $\D$ is dense, because one can consider functions of the form $$\chi_{\left\{x\in M : -n\leq \textup{Re}(\phi(x))\leq n, -m\leq \textup{Im}(\phi(x))\leq m\right\}},\quad m,n\in\bn, M\subseteq X \mbox{ measurable,}$$
and they are in the domain $\D$ and span $L^2(\mu)$. Thus the operator $A$ is densely defined. 

Consider now $g\in\D(A^*)$. Then the map $\D\ni f\mapsto \ip{Af}{g}=:\varphi_g(f)$ is a bounded linear functional, i.e., 
$$\left|\int \phi f\cj g\,d\mu\right|\leq C\|f\|,\quad (f\in\D).$$
This implies that $\varphi_g$ can be extended continuously to the whole space $L^2(\mu)$ and therefore there exists an element $A^*g\in L^2(\mu)$ such that $\varphi_g(f)=\ip{f}{A^*g}$ for all $f\in \D$. Then 
$$\int f\cj{\cj \phi g}\,d\mu=\int f\cdot A^*g\,d\mu,\quad (f\in\D).$$
But this means that $A^*g=\cj \phi g$ a.e.

Conversely, if $g\in\D$ we want to see that $g\in\D(A^*)$. But, if $\phi g\in L^2(\mu)$, then $\cj \phi g\in L^2(\mu)$, so $g\in \D(A^*)$.

\end{proof}

\begin{corollary}\label{corfu3.1}
Let $\Lambda$ be some nonempty index set. The multiplication operator $$M_{I}(a_\lambda)_{\lambda\in\Lambda}=(\lambda a_\lambda)_{\lambda\in\Lambda},\quad \D(M_I)=\left\{(a_\lambda)_{\lambda\in\Lambda}\in l^2(\Lambda) : \sum_{\lambda\in\Lambda}|\lambda|^2|a_\lambda|^2<\infty\right\},$$
is self-adjoint. 

\end{corollary}

\begin{proof}
Using Lemma \ref{lemfu3} for the space $\Lambda$ with the discrete measure and $\phi(\lambda)=\lambda\in\br$, we get that the operator $M_I$ is self-adjoint. 
\end{proof}

Next, we recall some of the main ideas in Fuglede's paper \cite{Fug74}.
\begin{definition}\label{deffu}
For an open set $\Omega$ in $\br^d$, consider the partial differential operators $\frac{1}{2\pi i}\frac{\partial}{\partial x_j}$, $j=1,\dots,d$, defined on the space of infinitely differentiable functions with compact support contained in $\Omega$, $C_c^\infty(\Omega)$. These operators are symmetric (by integration by parts). 

\begin{itemize}
	\item[$\bullet$] We say that $\Omega$ has the {\it extension property} if there are commuting self-adjoint extension operators $H_j$, i.e., 
$$\frac{1}{2\pi i}\frac{\partial}{\partial x_j}\subseteq H_j, \quad j=1,\dots,d.$$

	\item[$\bullet$] {\it Commutativity} for the extension operators $H_j$ is in the strong sense of spectral resolutions. More precisely, all projections associated to the spectral resolutions of the operators $H_j$ must commute.  
	
	\item[$\bullet$] A set $\Omega$ of finite Lebesgue measure is called {\it spectral} if there exists a set of frequencies $\Lambda\subset\br^d$, such that the family of exponential functions $\{e_\lambda(x)=e^{2\pi i\lambda\cdot x} : \lambda\in\Lambda\}$ forms an orthogonal basis for $L^2(\Omega)$. The set $\Lambda$ is called a {\it spectrum} for the set $\Omega$. 
	
	\item[$\bullet$] A set $\Omega$ of finite Lebesgue measure is said to {\it tile $\br^d$ by translations} if there exists a set of vectors $\Gamma\subset\br^d$ such that the translates $\{\Omega+\gamma :\gamma\in\Gamma\}$ cover $\br^d$ up to measure zero, and if the intersections $(\Omega+\gamma)\cap(\Omega+\gamma')$ have measure zero for $\gamma\neq \gamma'$ in $\Gamma$. 
	
\end{itemize}

\end{definition}

\begin{remark}
In general, when $\Omega$ is given, the individual symmetric operators will have self-adjoint extensions, but the added condition that there is a choice of $d$ {\it mutually commuting} self-adjoint extensions is a strong restriction. For example, if $d=2$, and if  $\Omega$ is a triangle or a disk, then there will not be commuting self-adjoint extensions (see \cite{Fug74}). This point is clarified in the next theorem:
\end{remark}

\begin{theorem}\label{thfug}\cite{Fug74,Jor82,Ped87,JoPe92,JoPe00} Let $\Omega\subset\br^d$ be open and connected, with finite and positive Lebesgue measure. Then $\Omega$ has the extension property if and only if it is a spectral set. Moreover, with $\Omega$ given, there is a one-to-one correspondence between the two sets of subsets:
\begin{equation}
\{\Lambda\subset\br^d : \Lambda\mbox{ is a spectrum for }\Omega\}
\label{eqjp2.4}
\end{equation}
and 
\begin{multline}
\left\{\Lambda \subset\br^d :\Lambda\mbox{ is the joint spectrum of some commutative family }\right.
\label{eqjp2.5}\\
\left.(H_1,\dots,H_d) \mbox{ of self-adjoint extensions}\right\}.
\end{multline}
This correspondence is determined as follows:
\begin{itemize}
	\item[(a)] If the extensions $(H_1,\dots,H_d)$ are given, then $\lambda\in\Lambda$ if and only if 
	\begin{equation}
	e_\lambda\in\bigcap_j\D(H_j).
	\label{eqjp2.6}
	\end{equation}
	\item[(b)] If, conversely, $\Lambda$ is a spectrum for $\Omega$ at the outset, then the ansatz \eqref{eqjp2.6} and 
	\begin{equation}
	H_je_\lambda=\lambda_j e_\lambda,\quad \lambda\in\Lambda
	\label{eqjp2.7}
	\end{equation}
	determine uniquely a set of commuting extensions. 
\end{itemize}

If $\Omega$ is only assumed open, the the spectral-set property implies the extension property, but not conversely. 

\end{theorem}

\begin{conjecture}\label{confug}{\bf The Fuglede Conjecture.}\cite{Fug74} A set $\Omega$ of finite Lebesgue measure is spectral if and only if it tiles $\br^d$ by translations.
\end{conjecture}

\begin{remark}\label{remf2}
 We note the following conclusions from Theorem \ref{thfug}: The link between the theorem and Conjecture \ref{confug} is as follows: when commuting self-adjoint extensions exist, then automatically the joint spectrum is purely discrete, and the corresponding eigenspaces will be one-dimensional. And they are necessarily orthogonal in $L^2(\Omega)$ and spanned by Fourier frequencies. The link to geometry is on account of the fact that, when commuting self-adjoint extensions exist for a given $\Omega$, then they generate a unitary representation $U$ of $\br^d$, with $U$ acting on $L^2(\Omega)$. But locally (i.e., in the interior of $\Omega$), $U$ will then act by translations.
\end{remark}

\section{Symmetric and self-adjoint extensions}

In this section we investigate the symmetric and the self-adjoint extensions of the momentum operator $\Ds$ on $C_c^\infty(\Omega)$, where $\Omega$ is a union of intervals as in Definition \ref{defi2.1}.
 
\begin{theorem}\label{th1.1}
The operator $\Ds$ is symmetric. The adjoint $\Ds^*$ has domain $\Dmax$ as in \eqref{eq1.1.1} and 
\begin{equation}
\Ds^* f=\frac{1}{2\pi i}f',\mbox{ for }f\in\D(\Ds^*)=\Dmax.
\label{eq1.1.2}
\end{equation}
The operator $\Ds$ on $C_c^\infty(\Omega)$ has a closed extension to $\mathscr D_0(\Omega)$ (see \eqref{eqd_0}), and, for $f$ in $\mathscr D_0(\Omega)$, we also have $\Ds f=\frac{1}{2\pi i}f'$. The adjoint of the operator $\Ds|_{\mathscr D_0(\Omega)}$ is the same as the one described above. The adjoint of $\Ds^*$ is 
\begin{equation}
\left(\Ds|_{C_c^\infty(\Omega)}\right)^{**}=\Ds|_{\mathscr D_0(\Omega)}=\Ds^*|_{\mathscr D_0(\Omega)}.
\label{eq1.1.3}
\end{equation}
\end{theorem}

\begin{proof}
Let $g\in \D(\Ds^*)$. By definition, this means that, for all $f\in C_c^\infty(\Omega)$, $\ip{\Ds f}{g}=\ip{f}{\Ds^*g}$ which means that 
$$\frac1{2\pi i}\int_\Omega f'(x)\cj g(x)\,dx=\int_\Omega f(x)\cj{\Ds^* g}(x)\,dx.$$

Define the function $\varphi(x):=\int_{\alpha_i}^x\Ds^* g(t)\,dt$, for all $x\in J_i$, $i\in\{1,\dots,n\}$. Then $\varphi$ is absolutely continuous and $\varphi'(x)=\Ds^*g(x)$ for almost every $x\in\Omega$. Then, using integration by parts, and the fact that $f|_{\partial\Omega}=0$, we have:
$$\frac{1}{2\pi i}\int_\Omega f'(x)\cj g(x)\,dx=\int_\Omega f(x)\cj{\varphi'(x)}\,dx=\int_{\partial\Omega}f\cj\varphi-\int_\Omega f'(x)\cj{\varphi(x)}\,dx=-\int_\Omega f'(x)\cj{\varphi(x)}\,dx.$$
Then
$$\int_{\Omega} f'(x)\cj{\left(\frac{1}{2\pi i}g(x)-\varphi(x)\right)}\,dx=0,\mbox{ for all }f\in C_c^\infty(\Omega).$$
This means that the function $\varphi-\frac{1}{2\pi i}g$ is orthogonal to the range of the operator $\Ds$. 

Next, we compute the orthogonal complement of the range of the operator $\Ds$. Note that, if $f\in C_c^\infty(\Omega)$, then $f|_{\partial\Omega}=0$, and $f(x)=\int_{\alpha_i}^x f'(t)\,dt$ for all $x\in J_i$, so $\int_{\alpha_i}^{\beta_i}f'(t)\,dt=f(\beta_i)=0$. This implies that, for every function $h\in L^2(\Omega)$ which is constant on each interval $J_i$, we have $\int_\Omega\Ds f(x)\cj h(x)\,dx=0$, so $h$ is 
orthogonal to the range of $\Ds$.

Conversely, let $h\in L^2(\Omega)$ be orthogonal to the range of $\Ds$. Fix $i\in\{1,\dots,n\}$. We will show that $h$ has to be constant on $J_i$.

First, we show that the range of $\Ds$ contains the functions $f\in C_c^\infty(J_i)$ with $\int_{\alpha_i}^{\beta_i}f(t)\,dt=0$. Let $f$ be such a function and let $\psi(x)=2\pi i\int_{\alpha_i}^xf(t)\,dt$, for $x\in J_i$, and $\psi(x)=0$ otherwise. Then, $\psi\in C_c^\infty(\Omega)$ and $\Ds\varphi=f$, thus $f$ is in the range of $\Ds$. 

Take now a function $f\in L^2(\Omega)$, which is zero outside $J_i$ and with $\int_{\alpha_i}^{\beta_i}f(t)\,dt=0$. One can approximate $f$ in $L^2$ by a sequence of functions $f_n$ in $C_c^\infty(J_i)$ with $\int_{\alpha_i}^{\beta_i}f(t)\,dt=0$, therefore the functions $f_n$ are in the range of the operator $\Ds$. It follows that $h$ is orthogonal to the functions $f_n$, hence to $f$. 

Now take a function $\tilde f\in L^2(\Omega)$ which is zero outside $J_i$. Define $f(x)=\tilde f(x)-\frac{1}{\beta_i-\alpha_i}\int_{\alpha_i}^{\beta_i}\tilde f(t)\,dt$, for $x\in J_i$, and $f(x)=0$ outside $J_i$. Then $\int_{\alpha_i}^{\beta_i}f(t)\,dt=0$. Then $\int_\Omega h(x)\cj f(x)\,dx=0$, which means that 
$$\int_\Omega \left(h(x)-\frac1{\beta_i-\alpha_i}\int_{\alpha_i}^{\beta_i}h(t)\,dt\right)\cj{\tilde f(x)}\,dx=\int_{\alpha_i}^{\beta_i}h(x)\cj{\tilde f(x)}\,dx-\frac{1}{\beta_i-\alpha_i}\int_{\alpha_i}^{\beta_i}h(x)\,dx\cdot\int_{\alpha_i}^{\beta_i}\cj{\tilde f(x)}\,dx$$
$$=\int_\Omega h(x)\cdot \cj{\left(\tilde f(x)-\frac1{\beta_i-\alpha_i}\int_{\alpha_i}^{\beta_i}\tilde f(t)\,dt\right)}\,dx=\int_\Omega h(x)\cj{f(x)}\,dx=0.$$

Since $\tilde f$ is arbitrary, it follows that the function $h(x)-\frac1{\beta_i-\alpha_i}\int_{\alpha_i}^{\beta_i}h(t)\,dt$ is zero a.e., on each interval $J_i$, which means that $h$ is constant a.e., on each interval $J_i$.

Returning to the computation of the domain of $\Ds^*$, we obtain that $\varphi-\frac{1}{2\pi i}g$ is constant on each interval $J_i$. But then since $\Ds^*g$ is in $L^2(\Omega)\subset L^1(\Omega)$, it follows that $\varphi(x)=\int_{\alpha_i}^x\Ds^*g\,dx$ is absolutely continuous, on each interval $J_i$, so $g$ is as well, and $\frac1{2\pi i }g'=\varphi'=\Ds^*g$ a.e. on $\Omega$. 

Conversely, if $g$ is absolutely continuous on each interval $J_i$, and $g'\in L^2(\Omega)$, then using integration by parts as above, we have 
$\ip{\Ds f}{g}=\ip{f}{\frac{1}{2\pi i}g'}$, and therefore $g\in \D(\Ds^*)$ and $\Ds^*g=\frac{1}{2\pi i}g'$.

Next, we prove that the operator $\Ds$ on $\D_0(\Omega)$ is closed. Take a sequence $\{f_n\}$ in $\D_0(\Omega)$ which converges in $L^2(\Omega)$ to some function $f$, and such that the sequence $\{\Ds f_n\}$ converges in $L^2(\Omega)$ to some other function $\frac{1}{2\pi i}g$. We want to prove that $f$ is in $\D_0(\Omega)$ and $f'=g$. 

Define $\varphi(x)=\int_{\alpha_i}^x g(t)\,dt=\ip{g}{\chi_{(\alpha_i,x)}}$, for all $x\in J_i$. Then, for $x\in \Omega$, $\varphi(x)$ is the limit of $\ip{f_n'}{\chi_{(\alpha_i,x)}}=\int_{\alpha_i}^xf_n'(t)\,dt=f_n(x)-f_n(\alpha_i)=f_n(x)$. Since $\{f_n\}$ converges to $f$ in $L^2(\Omega)$, we obtain that $f=\varphi$ a.e. This implies that $f$ is absolutely continuous, and $f'=\varphi'=g$ a.e. Since $\varphi(\alpha_i)=0$, it follows that $f(\alpha_i)=0$. Also,
$$f(\beta_i)=\varphi(\beta_i)=\int_{\alpha_i}^{\beta_i}f'(t)\,dt=\lim_n\int_{\alpha_i}^{\beta_i}f_n'(t)\,dt=\lim_n(f_n(\beta_i)-f_n(\alpha_i))=0.$$

To prove that the adjoint, of $\Ds|_{\D_0(\Omega)}$ is as before, the same arguments can be used. Since $A=\Ds|_{\mathscr D_0(\Omega)}$ is closed, $A^{**}=A$, see \cite[Corollary 1.8, page 305]{Con90}.
\end{proof}

\begin{theorem}\label{th2.2}
If $T$ is closed symmetric extension of $\Ds|_{C_c^\infty(\Omega)}$, then there exists a partial isometry $B$ between subspaces $B_l$ and $B_r$ of $\bc^n$ such that 
$$\D(T)=\left\{f\in \Dmax :  f(\vec \alpha)\in B_l, B f(\vec\alpha)= f(\vec \beta)\right\},$$
and $Tf=\Ds f$ for $f\in\D(T)$. Conversely, if the unbounded operator $T$ is defined as such, then it is a closed symmetric extension of $\Ds|_{C_c^\infty(\Omega)}$. 

The adjoint $T^*$ has domain 
$$\D(T^*)=\left\{f\in\Dmax : (B^*\vec f(\beta)-\vec f(\alpha))\perp B_l\right\},$$
and $T^*f=\Ds f$, for $f\in \D(T^*)$.

The operator $T$ is a self-adjoint extension if and only if $B$ is unitary, i.e., $B_l=B_r=\bc^n$.

\end{theorem}

\begin{proof}

Let $T$ be a closed symmetric extension of $\Ds$. Then, for all $f\in C_c^\infty(\Omega)$, and $g\in\D(T)$, we have 
$$\ip{f}{Tg}=\ip{Tf}{g}=\ip{\Ds f}{g}.$$
By Theorem \ref{th1.1}, this implies that $g\in\Dmax$ and $Tg=\Ds^* g=\Ds g$. Thus $\D(T)\subseteq\Dmax$ and $Tg=\Ds g$, for all $g\in\D(T)$, in other words, $\Ds$ on $\Dmax$ is an extension of $T$. 

Using integration by parts we have, for all $f,g\in\Dmax$:
$$\ip{f}{\Ds g}=\int_\Omega f\cj{\frac{1}{2\pi i}g'}=-\frac{1}{2\pi i}\int_\Omega f\cj{g'}=-\frac{1}{2\pi i}\left(-\int_\Omega f'\cj g+\sum_{i=1}^n\left(f(\beta_i)\cj g(\beta_i)-f(\alpha_i)\cj g(\alpha_i)\right) \right)
$$
$$=\frac{1}{2\pi i}\int_\Omega f' g-\frac1{2\pi i}\left(\ip{ f(\vec\beta)}{ g(\vec\beta)}-\ip{ f(\vec\alpha)}{ g(\vec\alpha)}\right),
$$
(see the notation in Remark \ref{rem1.1}).
Thus, we have 

\begin{equation}
\ip{f}{\Ds g}=\ip{\Ds f}{g}-\frac1{2\pi i}\left(\ip{ f(\vec\beta)}{ g(\vec\beta)}-\ip{ f(\vec\alpha)}{ g(\vec\alpha)}\right),\mbox{ for all }f,g\in\Dmax.
\label{eqparts}
\end{equation}

Now, if $T$ is a closed symmetric extension of $\Ds$ on $C_c^\infty(\Omega)$, then, for $f,g\in \D(T)$, we have, $f,g\in\Dmax$ and 
$$0=\ip{f}{T g}-\ip{Tf}{g}=\ip{f}{\Ds g}-\ip{\Ds f}{g}=-\frac1{2\pi i}\left(\ip{ f(\vec\beta)}{ g(\vec\beta)}-\ip{f(\vec\alpha)}{g(\vec\alpha)}\right),$$
therefore 
\begin{equation}
\ip{f(\vec\alpha)}{ g(\vec\alpha)}=\ip{ f(\vec\beta)}{ g(\vec\beta)},\mbox{ for all }f,g\in\D(T).
\label{eq2.1.1}
\end{equation}
Taking $f=g$ in \eqref{eq2.1.1}, we get that $\| f(\vec\alpha)\|^2=\| f(\vec\beta)\|^2$ for all $f\in \D(T)$; in particular, if $ f(\vec\alpha)=0$, then $ f(\vec\beta)=0$. This implies that, if $f_1,f_2\in\D(T)$ and $f_1(\vec\alpha)= f_2(\vec\alpha)$ then $(f_1-f_2)(\vec\alpha)=0$ so $(f_1-f_2)(\vec\beta)=0$, and $ f_1(\vec\beta)= f_2(\vec\beta)$. This means that there exists a well defined function $B$, 
$$B( f(\vec\alpha))= f(\vec\beta),\mbox{ for all }f\in\D(T),\quad B:B_l\rightarrow B_r,$$
where 
$$B_l:=\{ f(\vec\alpha) : f\in\D(T)\},\quad B_r:=\{ f(\vec\beta) : f\in\D(T)\}.$$

In addition $B$ is a linear isometry between the subspaces $B_l$ and $B_r$ of $\bc^n$, and, by definition, $ f(\vec\beta)=B  f(\vec\alpha)$, for $f\in\D(T)$. We can define $B$ to be zero on the orthogonal complement of $B_l$.

Because of this, we obtain that $\D(T)$ is contained in 
$$\D_B:=\left\{f\in\Dmax : f(\vec\alpha)\in B_l, f(\vec\beta)\in B_r, B f(\vec\alpha)= f(\vec\beta)\right\}.$$

We prove that the reverse inclusion also holds. Let $f\in\D_B$. Then $ f(\vec\alpha)\in B_l$. Then, there exists $f_0\in \D(T)$ such that $ f_0(\vec\alpha)=f(\vec\alpha)$. Then also $ f_0(\vec\beta)=B f_0(\vec\alpha)=B f(\vec\alpha)=f(\vec\beta)$. Hence $(f-f_0)(\vec\alpha)=(f-f_0)(\vec\beta)=0$, and so $f-f_0\in\D_0(\Omega)\subseteq\D(T)$. Then $f=(f-f_0)+f_0\in\D(T)$.

Assume now, conversely, that we are given an partial isometry $B$ from $B_l$ to $B_r$, and we prove that $\Ds$ on $\D_B$ is symmetric and closed. 

For symmetry, we use \eqref{eqparts}; for $f,g\in\D(T)$, we have 
$$\ip{f}{\Ds g}=\ip{\Ds f}{g}-\frac1{2\pi i}\left(\ip{ f(\vec\beta)}{g(\vec\beta)}-\ip{ f(\vec\alpha)}{ g(\vec\alpha)}\right)$$$$=\ip{\Ds f}{g}-\frac{1}{2\pi i}\left(\ip{B f(\vec\alpha)}{B g(\vec\alpha)}-\ip{ f(\vec\alpha)}{ g(\vec\alpha)}\right)=\ip{\Ds f}{g}.$$ 

To see that the operator is closed, take $\{f_n\}$ in $\D(T)$ convergent to $f$ in $L^2(\Omega)$, and $\{f_n'\}$ convergent to $g$ in $L^2(\Omega)$. For $i\in\{1,\dots,n\}$ and $x\in J_i$, we have 
$$f_n(x)=f_n(\alpha_i)+\int_{\alpha_i}^x f_n'(t)\,dt.$$
On the left hand side $\{f_n\}$ converges to $f$ in $L^2(\Omega)$; on the right hand side $\int_{\alpha_i}^x f_n'(t)\,dt$ converges to $\int_{\alpha_i}^x g(t)\,dt$ for all $x\in J_i$. Therefore we obtain that $\{f_n(\alpha_i)\}$ converges to $c_i=f(x)-\int_{\alpha_i}^x g(t)\,dt$, for a.e. $x$. But then 
$f(x)=c_i+\int_{\alpha_i}^x g(t)\,dt$ for a.e. $x$ and therefore $f'=g$ a.e., so the operator is closed.

Consider now the extension $T=\Ds$ on $\D_B$. We will compute its adjoint. For $g\in\D(T^*)$, we have that $\D_B\ni f\mapsto \ip{\Ds f}{g}$ is bounded. In particular, it is bounded on $\D_0(\Omega)$ so $g\in\Dmax$ and $T^* g=\Ds g$, $\ip{\Ds f}{g}=\ip{f}{T^*g}=\ip{f}{\Ds g}$. Then, with integration by parts \eqref{eqparts}, we have that 
$$\ip{ f(\vec\alpha)}{B^* g(\vec\beta)}=\ip{B f(\vec\alpha)}{ g(\vec\beta)}=\ip{ f(\vec\beta)}{ g(\vec\beta)}=\ip{ f(\vec\alpha)}{ g(\vec\alpha)},\mbox{ for all }f\in\D_B.$$

This implies that $\left( B^* g(\vec\beta)- g(\vec\alpha)\right)$ is orthogonal to the subspace $B_l$.

Conversely, if $g\in\Dmax$ and $\left( B^* g(\vec\beta)- g(\vec\alpha)\right)\perp B_l$, then from the previous computation, and using integration by parts \eqref{eqparts}, we get that, for $f\in \D_B$, 
$$|\ip{\Ds f}{g}|=|\ip{f}{\Ds g}|\leq \|f\|\|\Ds g\|,$$ and so the linear map $\D_B\ni f\mapsto \ip{\Ds f}{g}$ is bounded and $g\in\D(T^*)$.

Now, let's consider the case when the extension $T$ is self-adjoint. In this case $\D(T^*)$ is contained in $\D(T)$, and therefore, if $g\in\Dmax$ with $B^* g(\vec\beta)- g(\vec\alpha)$ orthogonal to $B_l$, then we have that $B g(\vec\alpha)= g(\vec\beta)$. We will prove that $B_l$ must be the entire space $\bc^n$. Suppose there exists a non-zero vector $v$ orthogonal to $B_l$. Then, since $B$ is a partial isometry, there exists a non-zero vector $w$ orthogonal to $B_r$. Let $g\in\Dmax$ such that $ g(\vec\alpha)=0$ and $ g(\vec\beta)=w$ (for example, make $g$ piecewise linear on the intervals $J_i$). Then $B^*g(\vec\beta)-\vec g(\alpha)=0-0=0\perp B_l$. Thus $g\in\D(T^*)=\D(T)$, and therefore $B g(\vec\alpha)= g(\vec\beta)$, which implies that $B(0)=w$, a contradiction. Thus, when the extension is self-adjoint, we get that $B_r=\bc^n$ and $B_l=\bc^n$ and $B$ is unitary.

For the converse, if $B$ is unitary and $B_l=\bc^n$, $B_r=\bc^n$, then, $g\in\D(T^*)$, if and only if $B^* g(\vec\beta)- g(\vec\alpha)$ is orthogonal to $B_l$ so it must be zero, i.e., $B^* g(\vec\beta)=\ g(\vec\alpha)$, which is equivalent to $B g(\vec\alpha)= g(\vec\beta)$. This means that $\D(T^*)=\D(T)=\D_B$.
\end{proof}

\begin{definition}\label{defboma}
For a self-adjoint extension $A$ of the operator $\Ds|_{C_c^\infty(\Omega)}$ as in Theorem \ref{th2.2}, we call the unitary matrix $B$, the {\it boundary matrix associated to $A$}.

\end{definition}

\section{Spectral decomposition}

Having a self-adjoint extension $A$ of the momentum operator $\Ds$, we can use the Spectral Theorem \ref{thspectral} to obtain a spectral resolution of the self-adjoint operator $A$. We present in this section an explicit description of this spectral resolution. 

\begin{definition}\label{defsp1}
For $\vec z=(z_1,z_2,\dots,z_n)\in\bc^n$ denote by $E(\vec z)$, the $n\times n$ diagonal matrix with entries $(e^{2\pi iz_1},e^{2\pi iz_2},\dots,e^{2\pi iz_n})$.
\end{definition}

\begin{theorem}\label{thsp1}
Let $A$ be a self-adjoint extension of the operator $\Ds|_{C_c^\infty(\Omega)}$ and let $B$ its unitary boundary matrix. Let $P$ be the spectral measure for the operator $A$, so 
$$A=\int_{\br}t\,dP(t)$$
Then the spectral measure is atomic, supported on the spectrum  
$$\sigma(A)=\left\{\lambda\in\bc : \det(I-E(\lambda\vec\beta)^{-1}BE(\lambda\vec\alpha))=0\right\}\subseteq\br$$
which is a discrete unbounded set. For $\lambda\in\sigma(A)$, the eigenspace $P(\{\lambda\})L^2(\Omega)$ has dimension at most $n$, and it consists of functions of the form 
$$f(x)= e^{2\pi  i \lambda x} \sum_{i=1}^n c_i\chi_{J_i}(x),\mbox{ where }c=(c_i)_{i=1}^n\in\bc^n\mbox{, and }BE(\lambda\vec\alpha)c=E(\lambda\vec\beta)c.$$
\end{theorem}

\begin{proof}

We begin with a proposition.

\begin{proposition}\label{prsp2}
Let $A$ be a self-adjoint extension as in Theorem \ref{thsp1}. Let $\lambda\in\bc$. The following statements are equivalent:
\begin{enumerate}
	\item $\lambda$ is in the resolvent set of $A$. 
	\item The operator $A-\lambda I$ is onto.
	\item The matrix $E(\lambda\vec\beta)^{-1}BE(\lambda\vec\alpha)-I$ is onto.
	\item The operator $A-\lambda I$ in one-to-one. 
	\item The matrix $E(\lambda\vec\beta)^{-1}BE(\lambda\vec\alpha)-I$ is one-to-one.
\end{enumerate}

\end{proposition}

\begin{proof}
We prove that (ii) and (iii) are equivalent. 

The operator $A-\lambda I$ is onto, means that for every $g\in L^2(\Omega)$, there exists $f\in \D_B=\D(A)$, such that 
$$\frac{1}{2\pi i}f'-\lambda f=g.$$
We solve this first order linear differential equation on each interval $J_i$ of $\Omega$. We have $f'-2\pi i \lambda f=2\pi ig$. Multiplying by the integrating factor $e^{-2\pi i\lambda x}$, we get 
$$\left(e^{-2\pi i\lambda t} f(t)\right)'=2\pi i g(t)e^{-2\pi i\lambda t}.$$
Integrating, we get the general solution 
$$f(x)=e^{2\pi i \lambda x}\left(2\pi i\int_{\alpha_i}^x g(t)e^{-2\pi i\lambda t}\,dt+c_i\right),$$
for some constant $c_i$, for all $x\in J_i$, and all $i\in\{1,\dots,n\}$.

Since $g$ is in $L^2(\Omega)$, we see that $f$ is absolutely continuous and $f'=2\pi i (\lambda f+g)$ is in $L^2(\Omega)$, which means that $f$ is in the domain $\Dmax$, and the only thing that we have to insure is that $B \vec f(\alpha)=\vec f(\beta)$, by picking the right constants $(c_i)$.

Let's see what the condition $B \vec f(\alpha)=\vec f(\beta)$ means. We have 
$$f(\alpha_i)=c_ie^{2\pi i\lambda \alpha_i},\quad f(\beta_i)=e^{2\pi i\lambda \beta_i}(A_i+c_i),$$
where $A_i=2\pi i\int_{\alpha_i}^{\beta_i} g(t)e^{-2\pi i\lambda t}\,dt$.

Let $\vec A=(A_i)_{i=1}^n$, $\vec c=(c_i)_{i=1}^n$. Note that, by varying $g$, any vector in $\bc^n$ can be obtained as $\vec A$. Then, the condition $ f(\vec\beta)=B f(\vec\alpha)$ is equivalent to 
$$ E(\lambda\vec\beta)(\vec A+\vec c)=BE(\lambda\vec\alpha)\vec c,$$
or, equivalently,
$$ \vec A=(E(\lambda\vec\beta)^{-1}BE(\lambda\vec\alpha)-I)\vec c.$$
This shows that, the operator $A-\lambda I$ is onto if and only if the matrix $E(\lambda\vec\beta)^{-1}BE(\lambda\vec\alpha)-I$ is onto.

Next, we prove that (iv) and (v) are equivalent. The operator $A-\lambda I$ is not one-to-one means that there exists a non-zero $f\in\D_B$ with $\vec f(\beta)=B\vec f(\alpha)$, such that $\frac{1}{2\pi i}f'=\lambda f$. Solving this differential equation on each interval $J_i$, we obtain that $f(x)=c_ie^{2\pi i\lambda x}$, for $x\in J_i$, for some constant $c_i$. Then, the relation $\vec f(\beta)=B\vec f(\alpha)$ implies that, $(c_i e^{2\pi i\lambda\beta_i})_{i=1}^n= B(c_i e^{2\pi i\lambda\alpha_i})_{i=1}^n$; with $\vec c:=(c_i)_{i=1}^n$, this can be rewritten as $(E(\lambda\vec\beta)^{-1}BE(\lambda\vec\alpha)-I)\vec c=0$. Thus the operator $A-\lambda I$ if and only if the matrix $E(\lambda\vec\beta)^{-1}BE(\lambda\vec\alpha)-I$ is not one-to-one.

Finally, the statements (iii) and (v) are equivalent because they refer to a square matrix in a finite dimensional space. 

Now, since $A$ is self-adjoint, so also closed, $\lambda\in\rho(A)$ if and only if $A-\lambda I$ is both one-to-one and onto; but these two properties are equivalent, so (i) is equivalent to all the other statements. 

\end{proof}

Returning to the proof of Theorem \ref{thsp1}, we see that $\lambda$ is in the spectrum of $A$ if and only if $\det(I-E(\lambda\vec\beta)^{-1}BE(\lambda\vec\alpha))=0$. This is an analytic function of $\lambda$, therefore the zero set is discrete and at most countable, unless the function is identically zero. Suppose by contradiction that it is so. Multiplying on the left by $\det(E(\lambda\beta))$ and on the right by $\det(E(\lambda\alpha)^{-1})$, we get that 
$$0=\det(E(\lambda\vec\beta)E(\lambda\vec\alpha)^{-1}-B)=\det(E(\lambda(\vec\beta-\vec\alpha))-B),
$$
for all $\lambda\in\bc$. Take $\lambda=it$, with $t>0$, and let $t\rightarrow\infty$ Then $0=\det(E(\lambda(\vec\beta-\vec\alpha))-B)$ converges to $\det(-B)$ which is not zero (since $B$ is unitary), and this is a contradiction. Thus the spectrum of $A$ is discrete and at most countable. 

From the proof of Proposition \ref{prsp2}, we see that the eigenspace $P(\{\lambda\})L^2(\Omega)$ is as in the statement of the theorem, and hence has dimension at most $n$. Since, by the Spectral Theorem the orthogonal sum of the eigenspaces spans the entire Hilbert space $L^2(\Omega)$, it follows that $\sigma(A)$ cannot be finite, and since it is discrete, it has to be unbounded. 

\end{proof}

\begin{theorem}\label{thsp3}
Let $A$ be as in Theorem \ref{thsp1} and let $\{\lambda_n: n\in\bz\}$ a list of the eigenvalues of $A$ repeated according to multiplicity. Let $\{\epsilon_n :n\in\bz\}$ be an orthonormal basis of eigenvectors for $A$, $A\epsilon_n=\lambda_n \epsilon_n$, for all $n\in\bz$. Then 
\begin{equation}
\D(A)=\left\{f=\sum_{n\in\bz}f_n\epsilon_n\in L^2(\Omega) : \sum_{n\in\bz}|\lambda_n|^2|f_n|^2<\infty\right\}.
\label{eqsp3.1}
\end{equation}

\end{theorem}

\begin{proof}
Let $g=\sum_{n\in\bz}a_n\epsilon_n\in L^2(\Omega)$ with $\sum_{n\in\bz}|\lambda_n|^2|a_n|^2<\infty$. Then, for $f\in\D(A)$, we have
$$\ip{Af}{g}=\ip{\Ds f}{g}=\sum_n\cj a_n\ip{\Ds f}{\epsilon_n}=\sum_n\cj a_n\ip{f}{\Ds^*\epsilon_{n}}=\sum_n\cj a_n\lambda_n\ip{f}{\epsilon_n}=\ip{f}{\sum_n \lambda_n a_n\epsilon_n}.$$
This implies that $g\in \D(A^*)=\D(A)$ and $Ag=A^*g=\sum_{n}\lambda_n a_n\epsilon_n$. Thus $\D_I:=\{\sum_n a_n\epsilon_n : \sum_n |\lambda_n|^2|a_n|^2<\infty\}$ is contained in $\D(A)$. 

But, by Corollary \ref{corfu3.1}, the diagonal operator $M_I(\sum_{n}a_n\epsilon_n)=\sum_n\lambda_n a_n\epsilon_n$ defined on $\D_I$ is self-adjoint. Since self-adjoints operators are maximally symmetric, it follows that $\D_I=\D(A)$.
\end{proof}

\section{The unitary group}

If we have a self-adjoint extension $A$ of the momentum operator $\Ds$, we can associate to it a one-parameter unitary group $U(t)=\exp(2\pi it A)$, $t\in\br$ as in Theorem \ref{thstone0}. In this section we present some basic properties of this unitary group and show that it acts as translations inside the intervals and it splits points at the endpoints, with probabilities given by the boundary matrix $B$.

\begin{theorem}\label{thu}
Let $A=\Ds$ on $\D_B$ a self-adjoint extension with boundary matrix $B$. Let $$U(t)=\exp{(2\pi i t A)},\quad (t\in\br),$$ be the associated one-parameter unitary group.
\begin{enumerate}
	\item The domain $\D_B$ is invariant for $U(t)$ for all $t\in\br$, i.e., if $f\in\Dmax$ with $B f(\vec\alpha)= f(\vec\beta)$, then $U(t)f\in\Dmax$ with $B(U(t)f)(\vec\alpha)=(U(t)f)(\vec\beta)$.
	\item Fix $i\in\{1,\dots,n\}$ and let $t\in\br$ such that $J_i\cap (J_i-t)\neq \ty$. Then, for $f\in L^2(\Omega)$, 
	
	\begin{equation}
(U(t) f)(x)=f(x+t),\mbox{ for a.e. }x\in J_i\cap (J_i-t).
	\label{equ1}
	\end{equation}

In particular,
\begin{equation}
(U(\beta_i-x) f)(x)=f(\beta_i),\, (U(\alpha_i-x)f)(x)=f(\alpha_i),\mbox{ for }f\in\D_B, x\in J_i.
\label{eq4.2.2}
\end{equation}
  \item For $f\in\D_B$, if $x\in J_i$ and $t>\beta_i-x$, then

\begin{equation}
\left[ U(t)f\right](x)=\pi_i\left(B\left[U(t-(\beta_i-x))f\right](\vec\alpha)\right).
\label{eq4.4.1}
\end{equation}
Here $\pi_i:\bc^n\rightarrow\bc$ denotes the projection onto the $i$-th component $\pi (x_1,x_2\dots,x_n)=x_i$.

\end{enumerate}
\end{theorem}

\begin{proof}
(i) This follows from more general rules, see Theorem \ref{thstone0}(c), but we include a more direct proof. With the notation as in Theorem \ref{thsp3}, $f=\sum_n f_n\epsilon_n$ is in the domain of $A$, if and only if $\sum_n|\lambda_n|^2|f_n|^2<\infty$. Then, for $t\in\br$, $U(t)f=\sum_n f_ne^{2\pi i\lambda_nt}\epsilon_n$, and $\sum_n|\lambda_n|^2|f_ne^{2\pi i\lambda_nt}|^2=\sum_n|\lambda_n|^2|f_n|^2<\infty$, which means that $U(t)f$ is also in the domain of $A$.

(ii) Let $v_\lambda$ be an eigenvector for $A$ with eigenvalue $\lambda$. Then, by Theorem \ref{thsp1}, we have 
$$v_\lambda(x)=\sum_{k=1}^n c_k\chi_{J_k}(x)e^{2\pi i\lambda x},\quad (x\in\Omega),$$
for some constants $c_k\in\bc$. Then, for $x\in J_i\cap (J_i-t)$, we have $x,x+t\in J_i$, and  
\begin{equation}\label{eq4.2.1}
(U(t)v_\lambda)(x)=e^{2\pi i\lambda t}v_\lambda(x)=\sum_{k=1}^nc_k\chi_{J_k}(x)e^{2\pi i\lambda(x+t)}=v_\lambda(x+t).
\end{equation}

Now let $f\in L^2(\Omega)$. One can find a sequence $\{f_n\}$ of finite linear combinations of eigenvectors, such that $\lim f_n=f$ in $L^2(\Omega)$. Then $\lim U(t)f_n=U(t)f$, passing to subsequences, we can assume in addition that $\{f_n\}$ converges to $f$ pointwise a.e. $\Omega$, and $U(t)f_n$ converges to $U(t)f$ pointwise a.e. in $\Omega$. Then, for a.e. $x\in J_i\cap (J_i-t)$, we have 
$$(U(t)f)(x)=\lim (U(t)f_n)(x)=\lim f_n(x+t)=f(x+t).$$
(Note that we used also the fact that translation by $t$ preserves measure zero sets.)

The first relation in \eqref{eq4.2.2} follows from \eqref{eq4.2.1}, by taking $t=\beta_i-x-\epsilon$ and letting $\epsilon\rightarrow 0$. Similarly for the second relation.

(iii) Indeed, we have, with (i) and (ii),
$$\left[ U(t)f\right](x)=\left[ U(\beta_i-x)U(t-(\beta_i-x)) f\right](x)=\left[U(t-(\beta_i-x))f\right](\beta_i)$$$$ =\pi_i\left(\left[U(t-(\beta_i-x))f\right](\vec\beta)\right)=\pi_i\left(B\left[U(t-(\beta_i-x))f\right](\vec\alpha)\right).$$

\end{proof}

\section{Spectral sets}

In this section we consider the case when $\Omega$ is a spectral set. We present various characterizations of this property in terms of the self-adjoint extensions of the momentum operator $\Ds$ and in terms of the associated unitary groups. 

\begin{definition}\label{deffu1}
Assume that $\Omega$ is a spectral set with spectrum $\Lambda$. Recall that $e_\lambda$ denotes the exponential function $e_\lambda(x)=e^{2\pi i\lambda x}$. In order to make the vectors $e_\lambda$ in $L^2(\Omega)$ of norm one, we renormalize the Lebesgue measure on $\Omega$ by $\frac{1}{|\Omega|}\,dx$ (or we can simply assume that $\Omega$ has measure 1).

 The {\it Fourier transform (associated to the spectrum $\Lambda$)} is the unitary operator 
\begin{equation}
\mathcal F_\Lambda:L^2(\Omega)\rightarrow l^2(\Lambda), \quad \mathcal F_\Lambda f=\left(\ip{f}{e_\lambda}\right)_{\lambda\in\Lambda}.
\label{eqfu1.1}
\end{equation}

Define also the unbounded operator of {\it multiplication by the identity function} on $l^2(\Omega)$: 
\begin{equation}
M_{I}(a_\lambda)_{\lambda\in\Lambda}=(\lambda a_\lambda)_{\lambda\in\Lambda},\quad \D(M_I)=\left\{(a_\lambda)_{\lambda\in\Lambda}\in l^2(\Lambda) : \sum_{\lambda\in\Lambda}|\lambda|^2|a_\lambda|^2<\infty\right\}.
\label{eqfu1.2}
\end{equation}

Define the {\it unitary group associated to $\Lambda$} on $L^2(\Omega)$, by 

\begin{equation}
U_\Lambda(t)\left(\sum_{\lambda\in\Lambda} a_\lambda e_\lambda\right)=\sum_{\lambda\in\Lambda}e^{2\pi i\lambda t}a_\lambda e_\lambda,\mbox{ for }\sum_{\lambda\in\Lambda}a_\lambda e_\lambda\in L^2(\Omega), t\in\br.
\label{eqfu1.3}
\end{equation}

\end{definition}

\begin{definition}\label{deffulo}
A {\it unitary group of local translations} on $\Omega$ is a strongly continuous one parameter unitary group $U(t)$ on $L^2(\Omega)$ with the property that, for any $f\in L^2(\Omega)$ and any $t\in\br$, 
\begin{equation}
(U(t)f)(x)=f(x+t)\mbox{ for a.e. }x\in\Omega\cap(\Omega-t).
\label{eqlo1}
\end{equation}

\end{definition}

\begin{remark}
Note the difference between the Definition \ref{deffulo}, and the property (ii) in Theorem \ref{thu}. The Definition \ref{deffulo} is a stronger condition, because it allows jumps between different intervals $J_i$ of $\Omega$. 
\end{remark}
\begin{definition}\label{deffu3.1}
A unitary boundary matrix $B$ is called {\it spectral} if, for every $\lambda\in\br$, the equation $BE_\lambda(\vec\alpha)c=E_\lambda(\vec\beta)c$, $c\in \bc^n$ has either only the trivial solution $c=0$, or only constant solutions of the form $c=\alpha (1,1,\dots,1)$, $\alpha\in\bc$. 

\end{definition}

\begin{theorem}\label{thfu1}
Assume the $\Omega$ is a spectral set with spectrum $\Lambda$. Define the unbounded operator $A$ on $L^2(\Omega)$ by 
$$A\left(\sum_{\lambda\in\Lambda} f_\lambda e_\lambda\right)=\sum_{\lambda\in\Lambda}\lambda f_\lambda e_\lambda,\quad\D(A)=\left\{f=\sum_{\lambda\in\Lambda}f_\lambda e_\lambda : \sum_{\lambda\in\Lambda}|\lambda|^2|f_\lambda|^2<\infty\right\}.$$
Then, the operator $A$ is conjugate to the multiplication operator $M_I$, by the Fourier transform, i.e., 
\begin{equation}
A=\mathcal F_\Lambda^{-1}M_I\mathcal F_\Lambda.
\label{eqfu2.1}
\end{equation}
The domain $\D(A)$ contains $\D_0(\Omega)$ and all functions $e_\lambda$, $\lambda\in\Lambda$, and $A$ is a self-adjoint extension of $\Ds|_{C_c^\infty(\Omega)}$ with the property that all eigenvectors are constant multiples of exponential functions $ce_\lambda$, $c\in\bc$, $\lambda\in \Lambda$.

Conversely, if there exists a self-adjoint extension $A$ of  $\Ds|_{C_c^\infty(\Omega)}$ with the property that all eigenvectors are constant multiples of exponential functions $ce_\lambda$, $c\in \bc$, $\lambda\in\br$, then $\Omega$ is spectral, with spectrum 

\begin{equation}
\Lambda:=\left\{\lambda\in\br : e_\lambda\in \D(A)\right\}=\sigma(A).
\label{eqfu2.2}
\end{equation}

\end{theorem}

\begin{proof}
Equation \eqref{eqfu2.1} follows from a direct computation. To see that $A$ is self-adjoint, it is enough to check that $M_I$ is self-adjoint, and this is a consequence of Corollary \ref{corfu3.1}.

Clearly, the exponential functions $e_\lambda$ are in the domain $\D(A)$. Let's check that $\D_0(\Omega)$ is also contained in $\D(A)$, and $Af=\Ds f$ for $f\in \D_0(\Omega)$. Let $f=\sum_\lambda f_\lambda e_\lambda\in \D_0(\Omega)$ and $\lambda\in\Lambda (\subseteq \br)$. Then,
$$c_\lambda:=\ip{\Ds f}{e_\lambda}=\ip{f}{\Ds^* e_\lambda}=\ip{f}{\frac{1}{2\pi i} e_\lambda'}=\ip{f}{\lambda e_\lambda}=\lambda\ip{f}{e_\lambda}=\lambda f_\lambda.$$
Therefore, $\sum_\lambda |\lambda|^2|f_\lambda|^2=\sum_\lambda |c_\lambda|^2=\|\Ds f\|^2<\infty$, and 
$$\Ds f=\sum_\lambda \lambda f_\lambda e_\lambda=A f.$$

This shows that $A$ is indeed a self-adjoint extension of $\Ds|_{C_c^\infty(\Omega)}$.

Now we check that all eigenvectors of $A$ are constant multiples of $e_\lambda$. This follows from the next Lemma, and the fact that self-adjoint operators are closed.

\begin{lemma}\label{lemfuei}
Let $A$ be a closed unbounded operator on a Hilbert space $H$. Assume that there exists an orthonormal basis of $H$, $\{e_\lambda\}_{\lambda\in\Lambda}$, where $\Lambda\subset \bc$ and $Ae_\lambda=\lambda e_\lambda$, for all $\lambda\in\Lambda$. Then every eigenvector $f=\sum_\lambda f_\lambda e_\lambda$ with the property that $\sum_\lambda|\lambda|^2|f_\lambda|^2<\infty$ is of the form $ce_\lambda$, for some $c\in\bc$, $\lambda\in\Lambda$.

\end{lemma}

\begin{proof}
 Suppose $f=\sum_\lambda f_\lambda e_\lambda$ is an eigenvector for $A$ with eigenvalue $\lambda_0$. Then 
$$Af=\lambda_0 f=\sum_\lambda \lambda_0f_\lambda e_\lambda.$$
On the other hand , we have that $\sum_{\mbox{ finite}} f_\lambda e_\lambda$ converges in $L^2(\Omega)$ to $\sum_\lambda f_\lambda e_\lambda$; also
$$A\left(\sum_{\mbox{ finite}} f_\lambda e_\lambda\right)=\sum_{\mbox{ finite}}\lambda f_\lambda e_\lambda \rightarrow \sum_{\lambda} \lambda f_\lambda e_\lambda,$$
because $\sum_{\lambda}|\lambda|^2|f_\lambda|^2<\infty$. Since $A$ is closed we get 
$$Af=A\left(\sum_\lambda f_\lambda e_\lambda\right)=\sum_\lambda \lambda f_\lambda e_\lambda.$$
This means that $\lambda_0 f_\lambda=\lambda f_\lambda$ for all $\lambda\in\Lambda$, which implies that either $\lambda_0=\lambda$ or $f_\lambda=0$. Thus all the coefficients $f_\lambda$ are zero except for $f_{\lambda_0}$, so $f=f_{\lambda_0}e_{\lambda_0}$. 
\end{proof}

For the converse in the Theorem \ref{thfu1}, assume $A$ is a self-adjoint extension of $\Ds|_{C_c^\infty(\Omega)}$ with the property that all eigenvectors are constant multiples of exponential functions. Then, with Theorem \ref{thsp1}, for $\lambda$ in the spectrum $\sigma(A)=:\Lambda$, the subspace $P(\{\lambda\})L^2(\Omega)$ is one-dimensional, spanned by $e_\lambda$. Then the set of all eigenvectors for $A$, $\{e_\lambda : \lambda\in\Lambda\}$ is an orthonormal basis for $L^2(\Omega)$. 
Moreover, if $e_\lambda\in \D(A)$ for some $\lambda\in\br$, then $Ae_\lambda=\frac{1}{2\pi i}e_\lambda'=\lambda e_\lambda$, so \eqref{eqfu2.2} holds.

\end{proof}

\begin{theorem}\label{thfu4}
Assume that $\Omega$ is spectral with spectrum $\Lambda$. Let $A=A_\Lambda$ be the self-adjoint extension of $\Ds|_{C_c^\infty(\Omega)}$ defined in Theorem \ref{thfu1}, and let $B=B_\Lambda$ be the unitary boundary matrix associated to this extension as in Theorem \ref{th2.2}. Then $B$ is a spectral boundary matrix and it is uniquely and well-defined by the conditions 
\begin{equation}
Be_\lambda(\vec \alpha)=e_\lambda(\vec \beta),\mbox{ for all }\lambda\in\Lambda.
\label{eqfu4.1}
\end{equation}
Moreover 
\begin{equation}
\textup{span}\{ e_\lambda(\vec \alpha) : \lambda\in\Lambda\}=\textup{span} \{e_\lambda(\vec\beta):\lambda\in\Lambda\}=\bc^n.
\label{eqfu4.2}
\end{equation}

Conversely, if there exists a spectral boundary matrix $B$, then $\Omega$ is spectral with spectrum 
\begin{equation}
\Lambda=\{\lambda\in\br : Be_\lambda(\vec\alpha)=e_\lambda(\vec\beta)\}.
\label{eqfu4.3}
\end{equation}
\end{theorem}

\begin{proof}
By Theorem \ref{thfu1}, the only eigenvectors of the operator $A$ are multiples of $e_\lambda$, $\lambda\in\Lambda$. On the other hand, by Theorem \ref{thsp1}, the eigenvectors are functions of the form 
$$f=\sum_{i=1}^n\left(c_i\chi_{J_i}\right)e_\lambda,$$
where $c\in \bc^n$, with $BE_\lambda(\vec\alpha)c=E_\lambda(\vec\beta)c$. Thus, if  $BE_\lambda(\vec\alpha)c=E_\lambda(\vec\beta)c$ for some non-zero $c\in\bc^n$, then $f$ is an eigenvector, but then it must be a constant multiple of $e_\lambda$ so all the components of $c$ are the same. This means that $B$ is a spectral boundary matrix. 

We check now that the relation \eqref{eqfu4.1} completely determines $B$ as a well-defined unitary matrix. Since $\Lambda$ is a spectrum, for $\lambda\neq\lambda'$ in $\Lambda$ we have
$$0=\ip{e_\lambda}{e_{\lambda'}}=\sum_{k=1}^n\int_{\alpha_k}^{\beta_k}e^{2\pi i(\lambda-\lambda')x}\,dx=\sum_{k=1}^n\frac{1}{2\pi i(\lambda-\lambda')}\left(e^{2\pi i(\lambda-\lambda')\beta_k}-e^{2\pi i(\lambda-\lambda')\alpha_k}\right),$$
which means that 

\begin{equation}
\ip{e_\lambda(\vec\alpha)}{e_{\lambda'}(\vec\alpha)}=\ip{e_\lambda(\vec\beta)}{e_{\lambda'}(\vec\beta)},\mbox{ for all }\lambda,\lambda'\in\Lambda.
\label{eqfu4.4}
\end{equation}

Define the linear operator $B$ from $\textup{span}\{ e_\lambda(\vec \alpha) : \lambda\in\Lambda\}$ to $\textup{span} \{e_\lambda(\vec\beta):\lambda\in\Lambda\}$, by 
$$B\left(\sum_{\lambda\in\Lambda}a_\lambda e_\lambda(\vec\alpha)\right)=\sum_{\lambda\in\Lambda}a_\lambda e_\lambda(\vec\beta),$$
where only finitely many coefficients are non-zero. 

The operator is well-defined, because, if $\sum_\lambda a_\lambda e_\lambda(\vec\alpha)=0$, then 
$$0=\ip{\sum_\lambda a_\lambda e_\lambda(\vec\alpha)}{\sum_\lambda a_\lambda e_\lambda(\vec\alpha)}=\sum_{\lambda,\lambda'}a_\lambda\cj{a}_{\lambda'}\ip{e_\lambda(\vec\alpha)}{e_{\lambda'}(\vec\alpha)}
$$$$=\sum_{\lambda,\lambda'}a_\lambda\cj{a}_{\lambda'}\ip{e_\lambda(\vec\beta)}{e_{\lambda'}(\vec\beta)}=\ip{\sum_\lambda a_\lambda e_\lambda(\vec\beta)}{\sum_\lambda a_\lambda e_\lambda(\vec\beta)}$$
A similar argument shows that $B$ is unitary.

Next we prove that $B_l:=\textup{span}\{e_\lambda(\vec\alpha):\lambda\in\Lambda\}=\bc^n$. Since $B$ is unitary, the same will be true for $\vec\beta$. We proceed by contradiction, if $B_l$ is not the entire space $\bc^n$ then it has a non-trivial orthogonal complement, and therefore we can extend the partial isometry $B$, in two different ways to unitaries $\tilde B$ and $\tilde B'$. Both of them give rise, by Theorem \ref{th2.2}, to self-adjoint extensions of $\Ds|_{C_c^\infty(\Omega)}$, and, since $Be_\lambda(\vec\alpha)=e_\lambda(\vec\beta)$, we have $e_\lambda\in\D_{\tilde B}\cap \D_{\tilde B'}$. Then, with Lemma \ref{lemfuei}, all eigenvectors are multiples of $e_\lambda$, $\lambda\in\Lambda$, and with Theorem \ref{thsp3}, we get that 
$$\D_{\tilde B}=\left\{\sum_\lambda f_\lambda e_\lambda : \sum_\lambda |\lambda|^2|f_\lambda|^2<\infty\right\}=\D_{\tilde B'}.$$
But this means that $\tilde B=\tilde B'$, a contradiction, and \eqref{eqfu4.2} follows.

For the converse, if a spectral unitary boundary matrix is given, then the self-adjoint extension associated to it has all eigenvectors of the form $c e_\lambda$, with $c\in\bc$ and $\lambda\in\Lambda$ as in \eqref{eqfu4.3}, and they form an orthonormal basis. Thus $\Omega$ is spectral. 
\end{proof}

\begin{theorem}\label{thfu5}
Assume that $\Omega$ is spectral with spectrum $\Lambda$. Then the unitary group $U=U_\Lambda$ associated to $\Lambda$ is a unitary group of local translations. In addition, if $A$ is the self-adjoint extension associated to $\Lambda$ as in Theorem \ref{thfu1}, then 
\begin{equation}
U(t)=\exp(2\pi i tA),\quad (t\in\br).
\label{eqfu5.1}
\end{equation}
Conversely, if there exists a unitary group of local translations $(U(t))_{t\in\br}$ for $\Omega$, then $\Omega$ is spectral. 
\end{theorem}

\begin{proof}
Note first that $U(t)e_\lambda=e^{2\pi i\lambda t}e_\lambda$, for all $\lambda\in\Lambda$ and $t\in\br$. Then, for $t\in\br$ and $x\in\Omega\cap(\Omega-t)$, we have 
$$(U(t)e_\lambda)(x)=e^{2\pi i\lambda t}e^{2\pi i\lambda x}=e^{2\pi i\lambda(x+t)}=e_\lambda(x+t).$$

Now fix $t\in\br$ and let $f\in L^2(\Omega)$. We want to check \eqref{eqlo1}. Since $\{e_\lambda : \lambda\in\Lambda\}$ form an orthonormal basis for $L^2(\Omega)$, we can find a sequence of functions $\{f_n\}$ which are finite linear combinations of function $e_\lambda$, such that $\{f_n\}$ converges to $f$ in $L^2(\Omega)$. Passing to a subsequence, we can assume that $\{f_n\}$ converges to $f$ almost everywhere. Since $U(t)$ is unitary, $\{U(t)f_n\}$ converges to $U(t)f$ in $L^2(\Omega)$, and again passing to a subsequence we can assume in addition that $\{U(t)f_n\}$ converges to $U(t)f$ almost everywhere. 

We have $(U(t)f_n)(x)=f_n(x+t)$ for a.e. $x\in\Omega\cap(\Omega-t)$, for all $n$. Taking the limit $(U(t)f)(x)=f(x+t)$ for a.e $x\in \Omega\cap(\Omega-t)$. This proves \eqref{eqlo1} so $U$ is a unitary group of local translations. 

The relation \eqref{eqfu5.1} follows immediately, because we have the operators $A$ and $U(t)$ in diagonal form.

Assume now that $U(t)$ is a unitary group of local translations. By Stone's Theorem \ref{thstone}, there exists a self-adjoint operator $A$ such that $U(t)=\exp(2\pi itA)$. We claim that $A$ is a self-adjoint extension of $\Ds|_{C_c^\infty(\Omega)}$. We will prove that, for $f\in C_c^\infty(\Omega)$, 
\begin{equation}
\frac{1}{2\pi i t}(U(t)f-f)\mbox{ converges in $L^2(\Omega)$ to $\Ds f$ as $t\rightarrow 0$},
\label{eqfu5.2}
\end{equation}
which, by Theorem \ref{thstone0}(b) and (c), implies that $f$ is in the domain of $A$ and $Af=\Ds f$. This is to be expected, since, especially for small values of $t$, the operator $U(t)$ acts as a translation. 

We need a Lemma.

\begin{lemma}\label{lemfu6}
Assume that the function $f\in L^2(\Omega)$ is supported on the set $\Omega_\epsilon=\cup_{i=1}^n[\alpha_i+\epsilon,\beta_i-\epsilon]$ for some $\epsilon>0$, and that $|t|<\epsilon$. Then 
\begin{equation}
(U(t)f)(x)=f(x+t)\mbox{ for a.e. }x\in\Omega,
\label{eqfu6.1}
\end{equation}
where $f(x):=0$ for $x$ not in $\Omega$. 

\end{lemma}

\begin{proof}
If $|t|<\epsilon$, then $\Omega_\epsilon-t\subset (\Omega\cap(\Omega-t))$, and therefore $(U(t)f)(x)=f(x+t)$, for a.e. $x\in\Omega_\epsilon-t$. We prove that $g(x):=(U(t)f)(x)=0$ for $x\in\Omega\setminus(\Omega_\epsilon-t)$. We have
$$\|f\|_{L^2(\Omega)}^2=\|U(t)f\|_{L^2(\Omega)}^2=\int_{\Omega_\epsilon-t}|f(x+t)|^2\,dx+\int_{\Omega\setminus(\Omega_\epsilon-t)}|g(x)|^2\,dx$$
$$=\int_{\Omega_\epsilon}|f(x)|^2\,dx+\int_{\Omega\setminus(\Omega_\epsilon-t)}|g(x)|^2\,dx=\|f\|_{L^2(\Omega)}^2+\|g\|_{L^2(\Omega)}^2.$$
This implies that $g$ is 0 and we obtain the Lemma. 
\end{proof}

Take now $f\in C_c^\infty(\Omega)$. This means that $f$ is supported on a set $\Omega_\epsilon$ for some $\epsilon>0$. Using Lemma \ref{lemfu6}, for $|t|<\epsilon$ and for a.e. $x\in \Omega$, we have 
$$\frac{1}{2\pi it}((U(t)f)(x)-f(x))=\frac{1}{2\pi it}(f(x+t)-f(x)),$$
which converges uniformly to $\frac{1}{2\pi i}f'(x)$ (since $f\in C_c^\infty(\Omega)$) . Then \eqref{eqfu5.2} follows and therefore $A$ is a self-adjoint extension of $\Ds|_{C_c^\infty(\Omega)}$.

With Theorem \ref{thsp1}, the eigenvectors for $A$ are of the form $f=\left(\sum_{i=1}^n c_i\chi_{J_i}\right)e_\lambda$, $Af=\lambda f$. Then $U(t)f=e^{2\pi it\lambda}f$. 

Since $U(t)$ is a unitary group of local translations, for a.e. $x\in\Omega\cap(\Omega-t)$, we have 
$$e^{2\pi it\lambda}\left(\sum_{i=1}^n c_i\chi_{J_i}(x)\right)e_\lambda(x)=(U(t)f)(x)=f(x+t)=\left(\sum_{i=1}^n c_i\chi_{J_i}(x+t)\right)e_\lambda(x+t).$$

Fix $i\neq k$ in $\{1,\dots,n\}$ and choose $t$ such that $J_i\cap(J_k-t)\neq\ty$ and then pick $x\in J_i\cap(J_k-t)$ such that the previous relation holds. Then, we get 
$$e^{2\pi it\lambda}c_i e^{2\pi i\lambda x}=c_ke^{2\pi i\lambda(x+t)}.$$
This means that $c_i=c_k$, and thus $f=c_1 e_\lambda$. Therefore, all the eigenvectors of $A$ are multiples of $e_\lambda$, and, by Theorem \ref{thfu1}, it follows that $\Omega$ is spectral.
\end{proof}

\medskip

{\bf Concluding remarks.} As noted in the body of our paper, our present focus for the Fuglede conjecture is based on our particular choices of notions from geometry, harmonic analysis, and from spectral theory. Indeed, we have made these definite choices. Naturally, there are others, and readers will be able to review such alternative approaches in the literature; each one serving its purpose. As a guide to the relevant papers, we conclude here with the following list of citations \cite{Fug74,MR212589,MR4301821,MR4239181,MR3559001,MR3395232,MR3286496,MR4516178,MR4498477,MR4448683,MR1643694}.

The ideas in the present paper may be adapted to the case of scattering theory for quantum graphs. To see this, note that for the setting of the quantum graph case, then the intervals considered here would instead then represent edges in the quantum-graph network under consideration, see e.g., \cite{MR1204775,MR4616107}.

\medskip

{\bf Statements.} The paper does not generate data. The authors do not have any conflicts of interest.

%

\bibliographystyle{alpha}	
\bibliography{eframes}

\end{document}